\documentclass[12pt]{article}

\usepackage[margin=1in]{geometry}
\usepackage{amsmath,amssymb,amsthm}

\newtheorem{thm}{Theorem}[section]

\newtheorem{lemma}[thm]{Lemma}

\newcommand{\R}{\mathbb R}

\title{Borsuk--Ulam type theorem for the \\
orthogonal group and orthogonal four-partitions}
\author{Oleg R. Musin}
\date{}

\begin{document}
\maketitle

\begin{abstract}
Makeev~\cite{Mak07} stated that every finite Borel measure in $\mathbb R^d$ assigning
zero mass to hyperplanes can be cut by $d$ mutually orthogonal
hyperplanes so that every pair divides the measure into four equal
parts, and outlined a proof strategy, but the key steps were left
incomplete.  We give a direct and elementary proof.

The main ingredient is a Borsuk--Ulam theorem for $O(k)$: every
$B_k$-equivariant map from $O(k)$ to a natural representation of the
hyperoctahedral group $B_k$ has a zero.  An explicit model map has one
free orbit of zeros, consisting of the signed eigenbases of a diagonal
operator with simple spectrum.  A derivative computation and mod-$2$
equivariant degree complete the proof.  This is a self-contained proof
of a special case of the general Stiefel-manifold theorem in~\cite{Mus26}.
\end{abstract}


\section{Introduction}

The \emph{hyperoctahedral group} $B_k = A_k \rtimes S_k$ is the
semidirect product of the sign-change group
$A_k = (\mathbb{Z}/2)^k = \langle\lambda_1,\ldots,\lambda_k\rangle$
and the symmetric group $S_k$, where $S_k$ acts by
$\sigma(\lambda_i)=\lambda_{\sigma(i)}$.
The generator $\lambda_\ell \in A_k$ acts on $\mathbb{R}^k$ by
negating the $\ell$-th coordinate.

For each pair $\{i,j\}\subset[k]$ with $i<j$, let
$\mathbb{R}_{\lambda_i+\lambda_j}$ be the one-dimensional real
$A_k$-representation on which $\lambda_\ell$ acts by $-1$ if
$\ell\in\{i,j\}$ and by $+1$ otherwise.  Set
\[
 W_2^+ = \bigoplus_{1\le i<j\le k}\mathbb{R}_{\lambda_i+\lambda_j},
\]
extended to a $B_k$-representation by
$\sigma\,e_{\{i,j\}}=e_{\{\sigma(i),\sigma(j)\}}$.
One checks that $(W_2^+)^{B_k}=\{0\}$ and
$\dim W_2^+=\binom{k}{2}$.  Indeed, if a vector is fixed by
$\lambda_i$, every coefficient indexed by a pair containing $i$ must
equal its own negative; varying $i$ forces all coefficients to vanish.

\medskip

The Stiefel manifold $V_{k,k}=O(k)$ carries a free $B_k$-action:
$\lambda_\ell$ negates the $\ell$-th frame vector and $S_k$
permutes them.  Since $\dim W_2^+=\binom{k}{2}=\dim O(k)$,
the following zero theorem has the right dimensional setup.

\begin{thm}\label{main}
Every continuous $B_k$-equivariant map $f:O(k)\to W_2^+$ has a zero.
\end{thm}

\medskip

The classical ham-sandwich theorem~\cite{ST} guarantees that any
$d$ measures in $\mathbb{R}^d$ can be simultaneously bisected by a
single hyperplane.  For $d=2$, a single measure can already be
divided into four equal parts by two orthogonal lines~\cite{Shashkin}.

For general $d$, Makeev~\cite{Mak07} stated the following result.

\begin{thm}\label{mak}
Let $d\ge2$ and let $\mu$ be a finite Borel measure in $\mathbb{R}^d$
for which every hyperplane has measure zero.  Then there exist $d$
mutually orthogonal hyperplanes such that every pair of them divides
$\mu$ into four equal parts.
\end{thm}

Makeev outlines a dimension-count and parity argument, but leaves
without proof the key steps: uniqueness for the model example,
transversality, and the limiting argument.
Partial results were later obtained in~\cite{BK12,Mejia25}.
The present note gives a direct proof that supplies all three steps.

We deduce Theorem~\ref{mak} from Theorem~\ref{main} as follows.
For a measure $\mu$ with smooth positive density, each unit vector
$u$ determines a unique bisecting hyperplane, and the imbalance
of $\mu$ across each pair of hyperplanes defines a continuous
$B_d$-equivariant map $F:O(d)\to W_2^+$.
A zero of $F$ gives the desired orthogonal configuration;
the general case follows by Gaussian approximation.

The result is also a special case of the general orthogonal mass
partition theorem of~\cite{Mus26}, whose proof uses characteristic
classes on Stiefel manifolds.  The proof here instead uses a single
explicit model map and mod-$2$ degree.

\section{Proofs}

\subsection{Proof of Theorem~\ref{main}}\label{sec:proof-main}

We use the following standard mod-$2$ form of equivariant degree.

\begin{lemma}[Equivariant parity principle]\label{lem:degree}
Let a finite group $G$ act freely and smoothly on a closed
$n$-manifold $M$, and let $V$ be an $n$-dimensional real
$G$-representation.  Suppose that a smooth $G$-equivariant map
$h:M\to V$ has precisely one $G$-orbit of zeros and that every zero
is nondegenerate.  Then every continuous $G$-equivariant map
$f:M\to V$ has a zero.
\end{lemma}

\begin{proof}
Equivariant maps $M\to V$ are sections of the rank-$n$ bundle
$E=M\times_GV\to M/G$.  The section induced by $h$ has one transverse
zero, so
$\langle w_n(E),[M/G]\rangle=1\in\mathbb F_2$.
Thus $w_n(E)\ne0$, whereas a nowhere-zero section would force
$w_n(E)=0$.  The conclusion follows, with continuous sections handled
by approximation.  This is also the mod-$2$ equivariant degree
principle; see~\cite[Theorem~1]{Mus12}.
\end{proof}

\begin{proof}[Proof of Theorem~\ref{main}]
By Lemma~\ref{lem:degree}, it is enough to construct a smooth model map
whose zero set is a single free $B_k$-orbit and whose zeros are
nondegenerate.  We verify the three required facts:
\textit{Equivariance}, \textit{Zero set}, and \textit{Nondegeneracy}.

Fix pairwise distinct numbers $\alpha_1<\cdots<\alpha_k$ and, in the
standard orthonormal basis $e_1,\ldots,e_k$ of $\mathbb{R}^k$, take the
diagonal self-adjoint operator
\[
 A=\operatorname{diag}(\alpha_1,\ldots,\alpha_k),
 \qquad Ae_i=\alpha_i e_i.
\]
Define
\[
 h:O(k)\longrightarrow W_2^+,
 \qquad
 h(u_1,\ldots,u_k)
 =\sum_{1\le i<j\le k}
 \langle Au_i,u_j\rangle\, e_{\{i,j\}}.
\]

\textit{Equivariance.}
Negating $u_\ell$ negates $\langle Au_i,u_j\rangle$ exactly when
$\ell\in\{i,j\}$, matching the sign action on $e_{\{i,j\}}$.
Permuting the frame vectors permutes the index pairs.
Hence $h$ is $B_k$-equivariant.

\textit{Zero set.}
Write $U$ for the orthogonal matrix with columns $u_1,\ldots,u_k$.
The matrix of $A$ in this basis is $U^TAU$, and its $(i,j)$-entry is
\[
 (U^TAU)_{ij}=\langle Au_j,u_i\rangle.
\]
Since $A$ is self-adjoint, this matrix is symmetric and
$\langle Au_j,u_i\rangle=\langle Au_i,u_j\rangle$.  Consequently,
$h(U)=0$ exactly when every off-diagonal entry of $U^TAU$ vanishes,
that is, exactly when $U^TAU$ is diagonal.  In that case each column
$u_i$ is an eigenvector of $A$.  The spectrum is simple, so every
eigenspace is one-dimensional; hence
\[
 u_i=\varepsilon_i e_{\sigma(i)}
\]
for some signs $\varepsilon_i\in\{\pm1\}$ and a permutation
$\sigma\in S_k$.  Conversely, every signed permutation of the standard
basis is clearly a zero.  Therefore
$Z_h=B_k\cdot(e_1,\ldots,e_k)$ is a single free $B_k$-orbit.

\textit{Nondegeneracy.}
Identify tangent vectors to $O(k)$ at the eigenframe
$U_0=(e_1,\ldots,e_k)$ with skew-symmetric matrices $X=(x_{ij})$.
This follows by differentiating $U(t)^TU(t)=I$.  Along
$U(t)=U_0\exp(tX)$, the $i$-th column satisfies
\[
 u_i'(0)=\sum_p x_{pi}e_p.
\]
For $i<j$, the product rule gives
\begin{align*}
 (Dh)_{U_0}(X)_{\{i,j\}}
 &=\langle Au_i'(0),e_j\rangle
   +\langle Ae_i,u_j'(0)\rangle\\
 &=\alpha_jx_{ji}+\alpha_ix_{ij}
  =(\alpha_i-\alpha_j)x_{ij},
\end{align*}
because $x_{ji}=-x_{ij}$.  The variables $x_{ij}$, $i<j$, are
coordinates on $T_{U_0}O(k)$, and the same pairs index a basis of
$W_2^+$.  Thus $(Dh)_{U_0}$ is diagonal with nonzero diagonal entries
$\alpha_i-\alpha_j$.  With the coordinate pairs ordered in the same way,
\[
 \det (Dh)_{U_0}=\prod_{1\le i<j\le k}(\alpha_i-\alpha_j)\ne0.
\]
Hence $(Dh)_{U_0}$ is an isomorphism, so $U_0$ is an isolated zero and
is nondegenerate.  Equivariance gives the same conclusion at every
other zero.  Thus $Z_h$ is one free orbit of nondegenerate zeros, and
Lemma~\ref{lem:degree} shows that every continuous
$B_k$-equivariant map $O(k)\to W_2^+$ has a zero.
\end{proof}

\subsection{Proof of Theorem~\ref{mak}}\label{sec:proof-mak}

Throughout, a \emph{measure} is a finite nonzero Borel measure on
$\R^d$ assigning zero mass to every hyperplane.

\begin{proof}
We begin with the main idea.  For every unit vector $u$, choose the
hyperplane normal to $u$ that
bisects the measure.  An orthonormal frame then gives $d$ mutually
orthogonal bisecting hyperplanes.  For each pair, we record the signed
imbalance between the two opposite pairs of quadrants.  These
$\binom d2$ numbers form a $B_d$-equivariant map
$F:O(d)\to W_2^+$.  Theorem~\ref{main} supplies a frame at which all
imbalances vanish, and the bisection conditions then force all four
quadrants to have equal measure.  We first carry this out for a smooth
positive density and then pass to a general measure by approximation.

\textit{Smooth case.}
Assume first that $\mu$ has a smooth strictly positive density.
For each unit vector $u$, there is a unique bisecting hyperplane
$H(u)=\{x:\langle x,u\rangle=t(u)\}$; set
$H^+(u)=\{x:\langle x,u\rangle\ge t(u)\}$ and
$H^-(u)=\{x:\langle x,u\rangle\le t(u)\}$.  The assignment $u\mapsto t(u)$ is
continuous and odd.  Indeed, for fixed $u$ the distribution function
$s\mapsto\mu\{x:\langle x,u\rangle\le s\}$ is continuous and strictly
increasing from $0$ to $\mu(\R^d)$, giving existence and uniqueness.
The offsets remain bounded because every bisecting hyperplane meets a
fixed ball containing more than half of the mass; dominated convergence
and uniqueness then give continuity in $u$.  Reversing $u$ exchanges the
two half-spaces, so $t(-u)=-t(u)$.

Define $F:O(d)\to W_2^+$ by
\[
 \begin{aligned}
 F(U)_{ij}={}&\mu\bigl(H^+(u_i)\cap H^+(u_j)\bigr)
            -\mu\bigl(H^+(u_i)\cap H^-(u_j)\bigr)\\
            &-\mu\bigl(H^-(u_i)\cap H^+(u_j)\bigr)
            +\mu\bigl(H^-(u_i)\cap H^-(u_j)\bigr).
 \end{aligned}
\]
Reversing $u_i$ or $u_j$ negates $F_{ij}$, and permuting frame
vectors permutes coordinates, so $F$ is $B_d$-equivariant.  Continuity
follows from the continuity of the offsets and dominated convergence,
since the boundaries of the moving quadrants have measure zero.

By Theorem~\ref{main}, $F(U)=0$ for some $U$.  For each pair $i<j$,
let $a,b,c,e$ be the measures of the four quadrants cut by $H(u_i)$
and $H(u_j)$.  The bisection conditions give $a+b=a+c=\tfrac12\mu(\R^d)$,
so $b=c$.  The condition $F_{ij}=0$ gives $a+e=b+c=2b$, while
$a+b+c+e=\mu(\R^d)$ gives $a+e=\mu(\R^d)-2b$, so $b=c=\tfrac14\mu(\R^d)$
and $a=e=\tfrac14\mu(\R^d)$.  The frame $U$ is orthonormal, so the
$d$ hyperplanes are mutually orthogonal.

\textit{General case.}
Let $M=\mu(\R^d)$ and put $\mu_r=\mu*\gamma_{1/r}$.
Then $\mu_r$ has a smooth strictly positive density and
$\mu_r\Rightarrow\mu$.  For each $r$, the smooth case gives an
orthonormal frame $U_r=(u_{1,r},\ldots,u_{d,r})$ and corresponding
offsets $t_{1,r},\ldots,t_{d,r}$.

Choose $R$ such that $\mu(\operatorname{int}B_R)>M/2$.  By weak
convergence, $\mu_r(B_R)>M/2$ for all sufficiently large $r$.
Hence every bisecting
hyperplane for $\mu_r$ must meet $B_R$, and consequently
$|t_{i,r}|\le R$.  By compactness, after passing to a subsequence we
may assume that $U_r\to U=(u_1,\ldots,u_d)$ and $t_{i,r}\to t_i$ for
$1\le i\le d$.

We use the standard continuity-set principle: under weak convergence,
intersections of finitely many moving half-spaces have convergent
measures when their normals and offsets converge and the limiting
boundary has measure zero.  It applies to every quadrant here, since its
boundary lies in the union of two limiting hyperplanes and hence has
$\mu$-measure zero.

Every approximating quadrant has $\mu_r$-measure $M/4$.  Hence every
limiting quadrant has $\mu$-measure $M/4$.  Since $U$ is still an
orthonormal frame, the limiting hyperplanes are mutually orthogonal and
give the required four-partitions.
\end{proof}

\section{Concluding remarks and further directions}
For a nonempty subset \(I\subset[k]\), put
\[
 \lambda_I=\sum_{i\in I}\lambda_i, \qquad W_j^+=\bigoplus_{\substack{I\subset[k]\\ |I|=j}}
 \R_{\lambda_I}, \quad 1\le j\le k.
\]
Let \(e_I\) denote a basis vector of the summand
\(\R_{\lambda_I}\), with \(S_k\) acting by
\(\sigma e_I=e_{\sigma(I)}\).
Together with the $A_k$-action, this makes \(W_j^+\) a
\(B_k\)-representation.  In the present note we use only $j=2$,
for which $\dim W_2^+=\binom{k}{2}=\dim O(k)$.

The construction suggests a programme for more general orthogonal
partition problems.  After the individual bisection conditions have
been imposed, the requirement that every $n$-element subfamily divide
the measure into $2^n$ equal parts is encoded by a test map
\[
 V_{d,k}\longrightarrow \bigoplus_{j=2}^{n}W_j^+.
\]
The problem is to prove that the associated bundle over $V_{d,k}/B_k$
has nonzero mod-$2$ Euler class, or, in the dimension-matched case, to
construct a model map with an odd number of free $B_k$-orbits of zeros.
The diagonal-operator construction solves the case $n=2$, $d=k$.
For $j\ge3$, generic symmetric tensors need not admit an orthogonal
diagonalization, so new constructions or direct characteristic-class
computations are required; see also~\cite{Mus26}.

\medskip
\noindent\small Oleg R. Musin,
University of Texas Rio Grande Valley,
Brownsville, TX 78520, USA.\\
\textit{email}: oleg.musin@utrgv.edu

\end{document}